\documentclass[12pt
]{amsart}

\usepackage[margin=1in]{geometry}

\usepackage{amsmath}
\usepackage{amsfonts}
\usepackage{amssymb}
\usepackage{amsthm}
\usepackage{setspace}
\usepackage{amsrefs}
\usepackage[usenames]{xcolor}

\usepackage{hyperref}
\hypersetup{hidelinks}

\title[strongly $\cc$-linear convexity]{Regularity of a $\db$-solution operator  for\\ strongly $\cc$-linearly convex domains with\\ minimal smoothness}
\author[X. Gong]{Xianghong Gong$^{\dag}$}

\author[L. Lanzani]{Loredana Lanzani$^{\dag\dag}$}

 \address{Department of Mathematics,
 University of Wisconsin-Madison, Madison, WI 53706, U.S.A.}
 \email{gong@math.wisc.edu}

\address{Department of Mathematics,
Syracuse University, Syracuse, NY 13244-1150, U.S.A.}
\email{llanzani@syr.edu}

\thanks{$^{\dag}$Partially supported by a grant from the Simons
Foundation (award number:~505027). }

\thanks{${\dag\dag}$ Partially supported by the National Science Foundation (DMS 1504589; DMS-1901978)}

\newcommand{\dist}{\operatorname{dist}}

\newtheorem{thm}{Theorem}[section]
\newtheorem{cor}[thm]{Corollary}
\newtheorem{prop}[thm]{Proposition}
\newtheorem{lemma}[thm]{Lemma}

\theoremstyle{definition}

\newtheorem{exmp}[thm]{Example}

\newtheorem{rem}[thm]{Remark}

\renewcommand{\th}[1]{\begin{thm}\label{#1}}
\renewcommand{\eth}{\end{thm}}
\newcommand{\co}[1]{\begin{cor}\label{#1}}
\newcommand{\eco}{\end{cor}}
\renewcommand{\le}[1]{\begin{lemma}\label{#1}}
\newcommand{\ele}{\end{lemma}}
\newcommand{\pr}[1]{\begin{prop}\label{#1}}
\newcommand{\epr}{\end{prop}}

\newcommand{\ga}{\begin{gather}}
\newcommand{\ega}{\end{gather}}
\newcommand{\gan}{\begin{gather*}}
\newcommand{\egan}{\end{gather*}}
\newcommand{\al}{\begin{align}}
\newcommand{\eal}{\end{align}}
\newcommand{\aln}{\begin{align*}}
\newcommand{\ealn}{\end{align*}}
\newcommand{\eq}[1]{\begin{equation}\label{#1}}
\newcommand{\eeq}{\end{equation}}

\newcommand{\f}[2]{\frac{#1}{#2}}

\newcommand{\ci}{~\cite}

\newcommand{\cc}{{\bf C}}
\newcommand{\nn}{{\bf N}}

\newcommand{\rr}{{\bf R}}

\newcommand{\ov}{\overline}

\newcommand{\RE}{\operatorname{Re}}
\newcommand{\IM}{\operatorname{Im}}
\renewcommand{\dbar}{\overline\partial}

\newcommand{\sign}{\operatorname{sign}}

\newcommand{\all}{\alpha}

\newcommand{\del}{\delta}

\newcommand{\var}{\varphi}
\newcommand{\e}{\epsilon}

\newcommand{\Om}{\Omega}

\newcommand{\la}{\lambda}

\newcommand{\pd}{\partial}

\newcommand{\re}[1]{(\ref{#1})}
\newcommand{\rea}[1]{$(\ref{#1})$}
\newcommand{\rl}[1]{Lemma~\ref{#1}}

\newcommand{\rp}[1]{Proposition~\ref{#1}}
\newcommand{\rt}[1]{Theorem~\ref{#1}}

\newcommand{\rrem}[1]{Remark~\ref{#1}}

\newcommand{\rpa}[1]{Proposition~$\ref{#1}$}
\newcommand{\rta}[1]{Theorem~$\ref{#1}$}

\newcommand{\db}{\dbar}

\newcounter{pp}
\newcommand{\bpp}{\begin{list}{$\hspace{-1em}\alph{pp})$}{\usecounter{pp}}}
\newcommand{\epp}{\end{list}}

\newcounter{ppp}
\newcommand{\bppp}{\begin{list}{$\hspace{-1em}(\roman{ppp})$}{\usecounter{ppp}}}
\newcommand{\eppp}{\end{list}}

\def\beq{\begin{equation}}
\def\eeq{\end{equation}}

\begin{document}

\begin{abstract}
We prove regularity of solutions of the $\bar\partial$-problem in the H\"older-Zygmund spaces of bounded, strongly $\mathbf C$-linearly convex domains of class $C^{1,1}$. The proofs rely on a new analytic characterization of said domains which is of independent interest, and on
  techniques that were recently developed by the first-named author to prove estimates
 for the
 $\bar\partial$-problem on strongly pseudoconvex domains of class $C^2$.
  \end{abstract}

 \maketitle
{\centerline{\em For Eli}}
\medskip

\setcounter{thm}{0}\setcounter{equation}{0}
\section{Introduction}\label{S:intro}
 Let $D$ be a bounded domain in $\cc^n$
 with a
defining function $r$ that is $C^1$-smooth on a neighborhood of $\ov D$.
We say that
$D$ is {\it strongly $\cc$-linearly convex} if
\eq{c+0}
|r_\zeta\cdot(\zeta-z)|\geq c|\zeta-z|^2, \quad z\in\ov D, \quad \zeta\in\pd D
\eeq
for a positive constant $c$ that may depend on $r$.

 We say that $D$   is {\it strictly $\cc$-linearly convex} if it satisfies the weaker condition
 \eq{cond-sC}
|\,
 r_\zeta\cdot(\zeta-z)|>0, \quad \zeta\in\pd D, \quad z\in {\ov D\setminus \{\zeta\}}.
\eeq

 We say that $D$   is {\it weakly}
 $\cc$-linearly convex\footnote{
 sometime simply referred to  as ``weakly linearly convex".}
 if it satisfies the even milder condition
 \eq{cond-wsC}
|\,
 r_\zeta\cdot(\zeta-z)|>0, \quad \zeta\in\pd D, \quad z\in D.
\eeq
 \vskip0.08in
The notion of strong (resp., strict; weak) $\cc$-linear convexity is essentially intermediate between strong (resp., strict; weak) convexity and strong (resp.
  weak) pseudoconvexity\footnote{while the notions of ``strong'' and ``strict'' convexity (resp., ``strong'' and ``strict'' $\cc$-linear convexity) are distinct from one another, there is no distinction between ``strong'' and ``strict'' pseudoconvexity and the two terms are often interchanged. See \cite{MR3084008}*{p. 261}.}; it was first introduced by Behnke and Peschl~\ci{MR1512986} in 1935 and has since played a central role in the theory of Hardy spaces and holomorphic singular integral operators. The purpose of this paper is to extend the analysis of these domains to the
 $\dbar$-problem. Strongly $\cc$-linearly convex domains of class $C^2$ are, in particular, strongly
 pseudoconvex;
  see \cite{MR3084008}.
 Thus
  the $\dbar$-problem for such domains is well understood; our main goal here is to go below the $C^2$-category and extend the theory of $\dbar$ to the class $C^{1,1}$.
  As is well known,
$\cc$-linearly convex domains support a Cauchy-Fantappi\'e kernel, the {\em Cauchy-Leray kernel}, that is holomorphic in the output variable $z\in D$ and is ``canonical'' in the sense that  it is independent of the choice of defining function $r$ (unlike other instances of Cauchy-Fantappi\`e kernels). The Cauchy-Leray kernel determines a singular integral operator
 (the {\em Cauchy-Leray transform} acting on functions supported in the topological boundary $\pd D$) whose operator-theoretic properties depend on the
 boundary regularity and on the amount of $\cc$-linear convexity enjoyed by the domain.
E. M. Stein and the second-named author
 have shown ~\ci{MR3250300} that the
Cauchy-Leray transform associated to a bounded, strongly $\cc$-linearly convex domain $D\subset\cc^n$ with boundary of class $C^{1,1}$, initially defined for functions in $C^1(\pd D)$, extends to a bounded operator:
$L^p(\pd D,\mu)\to L^p(\pd D,\mu)$ for $1<p<\infty$ and with $\mu$ belonging to a family of  boundary measures that includes
induced Lebesgue measure $\sigma$. The proof relies, among other things, on the analysis of the (suitably defined)
action of the complex Hessian of $r$ (assumed to be only of class $C^{1,1}$) over the complex tangent space of $\pd D$. In the subsequent paper \cite{LSEx} examples were supplied that indicate that the two hypotheses of strong $\cc$-linear convexity and class $C^{1,1}$ are essentially optimal.

 \vskip0.08in

In this paper we provide  integral formulas-based solutions to the $\dbar$-problem for bounded,
 strongly $\cc$-linearly convex domains of class $C^{1,1}$:
 we first construct
homotopy formulas based on
 a hierarchy of Cauchy-Leray-Koppelman kernels that give rise to integral operators acting on forms of type $(0, q)$, $q=0,\ldots, n$ with coefficients
 defined
  on $\ov D$, the closure of the ambient domain. From these we obtain
 new estimates in the H\"older and Zygmund spaces  that give the expected optimal gain of $1/2$ derivatives.  It turns out
  that
  the
   classical Leray-Koppelman
   homotopy formulas are in fact true under milder notions of $\cc$-linear convexity: this is the reason why
  we   mentioned condition
  \eqref{cond-wsC} above.
 Strong $\cc$-linear convexity is however needed
 both to justify the use of the regularity estimates that were obtained by the first-named author in ~\ci{MR3961327}, and to prove a new homotopy formula in
  this paper.

\vskip0.08in
Our proofs rely
on the
 following characterization, which is of independent interest,
  of strong $\cc$-linear convexity
 for domains whose boundary is assumed to be of class $C^{1,1}$.
\th{equiv-cond}Let $D$ be a bounded domain of class $C^{1,1}$, and let $r$ be any defining function for $D$ that is of class $C^{1,1}$ in a neighborhood of $\ov D$.
Then we have that condition \eqref{c+0}
being satisfied by $r$
 is equivalent to
the same
 $r$ satisfying each of the following:
\eq{C+}
|r_\zeta\cdot(\zeta-z)|\geq c_1|\zeta-z|^2\ \  \text{for}\ \  z\in\ov D\ \  \text{and}\ \ \zeta\in U\setminus D.
\eeq
\eq{b}
|r_\zeta\cdot(\zeta-z)|\geq c_2(r(\zeta)-r(z)+|\zeta-z|^2)\ \ \text{for}\ \ z\in \ov D\ \
 \text{and} \ \ \zeta\in U\setminus D.
\eeq
\eq{c}
|r_\zeta\cdot(\zeta-z)|\geq c_3|\zeta-z|^2\ \ \text{for}\ \ \zeta, z\in\pd D.
\eeq
\vskip0.05in

    Here $U$ is a  neighborhood of $\ov D$ which may not be the same for \eqref{C+} and \eqref{b}.
\eth

We obtain the following main results.
\th{gain1/2}Let $D\subset \cc^n$ be a bounded
domain of class $C^{1,1}$,
and suppose that $D$ has a
 $C^{1,1}$
defining function $r$ in a neighborhood $U$ of $\ov{D}$ with the property that
 \eq{strictconvU}
|r_\zeta\cdot(\zeta-z)|>0, \quad z\in D, \quad \zeta \in U\setminus D\, .
\eeq
Then there
exist   homotopy formulas
on $D$
\ga\label{htf}
\var=\db H_q\var+H_{q+1}\db\var\, ,\quad q=1,\ldots,
n-1;
\\
\var=H_0\var+H_{1}\db\var, \quad q=0
\label{ht0}
\end{gather}
for forms $\var$  of type $(0,q)$ satisfying the assumption that $\var$ and $\dbar\var$ have coefficients in $C^1(\ov D)$.

 Furthermore,
 if $D$ is strongly $\cc$-linearly convex, then for any $a\in(1,\infty)$ we have
\ga\label{theest}
|H_{q}\var|_{\Lambda^{a+1/2}(\ov D)}\leq C_a|\var|_{\Lambda^{a}(\ov D)}, \quad   \quad q\geq1,\\
|H_{0}\var|_{\Lambda^{a}(\ov D)}\leq C_a|\var|_{\Lambda^{a}(\ov D)}.
\label{thest++++}
\end{gather}
\eth
Here $\Lambda^\beta$ is the standard H\"older space when $\beta\in(0,\infty)\setminus \nn$, and  is the Zygmund space when $\beta$ is a positive integer. Note that  condition \eqref{strictconvU} is
trivially implied by e.g., condition
 \eqref{C+} in Theorem \ref{equiv-cond}.
 The operators $H_q$, $q=0, 1, \ldots n$, are defined in Proposition~\ref{hf},
where the homotopy formulas \eqref{htf} and \eqref{ht0} are obtained.  In the proof of the regularity estimates \eqref{theest} and \eqref{thest++++}
 it will be important to work with a
  particular
  defining function for $D$: it is well known that condition \eqref{c+0} and hence \eqref{c} are independent of the choice of defining function
 for $D$ in the sense that only the constants will be affected by the choice of $r$.
 In Section \ref{S:basics} we show that such stability is also satisfied by
conditions \eqref{C+} and \eqref{b}
 (see
 Lemma~\ref{leray-bound+}
 for the precise statement).
Condition  \eqref{b} for a specialized choice of  $r\in C^{1,1}(U)\cap C^\infty(U\setminus \ov D)$ is then needed to justify
the application
to such $r$ of the results~\cite{MR3961327}*{Propositions 4.4 and 4.10} which
 in turn give the estimates \eqref{theest}-\re{thest++++}.

 \vskip0.08in

A statement analogous to \rt{gain1/2} was proved by the first-named author~\ci{MR3961327} under the assumptions that the bounded domain $D$ is  strongly pseudoconvex and has boundary of class $C^2$. In \rt{gain1/2} we essentially {\em increase} the amount of
convexity to strong $\cc$-linear convexity, and {\em reduce} the amount of boundary regularity to the class $C^{1,1}$.

 \vskip0.08in

By employing condition \eqref{c+0} and adapting the method of proof of the classical $C^{1/2}$ estimate for strongly pseudoconvex domains of class $C^2$ (see~\cite{MR774049}*{
Theorem 2.2.2}), we also obtain
\pr{0gain1/2}Let $D\subset \cc^n$ be a bounded
weakly
$\cc$-linearly convex domain of class $C^{1,1}$.
There exists a homotopy formula on $D$
\eq{htf++++}
\var=\db T_q\var+T_{q+1}\db\var \, ,\quad q=1,\ldots,
n-1
\eeq
for forms  $\var$  of type $(0,q)$ when
  $\var$ and $\dbar\var$ have coefficients in $C^0(\ov D)$.
Furthermore,
  if $D$ is strongly $\cc$-linearly convex
  we have that
\eq{estTq}
|T_q\var|_{C^{1/2}(\ov D)}\leq C|\var|_{C^0(\ov D)},\quad q\geq1.
\eeq
\epr
Here $T_q$, $q=1,\ldots, n$ are the classical Leray-Koppelman operators \cite{MR1800297}*{p.~273}, which must be suitably interpreted when $D$ is merely of class $C^{1,1}$; see Section \ref{S:homotopy}
for the precise statements and the proofs.

 \vskip0.08in

Note that \rt{gain1/2} and \rp{0gain1/2} do not include data $\var$ of maximal type $(0, n)$ because such data can be
 treated with techniques already available in the literature.
Indeed it was observed by the first-named author~\ci{MR3961327} that if $\var$ has maximal type $(0,n)$ and $D$ is a bounded Lipschitz
domain, the solutions of $\db u =\var$ can be easily obtained  by extending $\var$ to a form with compact support in $\cc^n$. Here $\var$ is obviously $\db$-closed and no convexity of $D$ is required.
Then one obtains
solutions that gain one full derivative in  H\"older and Zygmund spaces.

 \vskip0.08in

\rt{gain1/2} and Proposition \ref{0gain1/2}
 effectively illustrate that
 from the point of view
   of the $\db$-problem with data in the H\"older-Zygmund spaces,
  strongly $\cc$-linearly convex $C^{1,1}$ domains behave like strongly pseudoconvex $C^2$ domains.
 On the other hand, this analogy may fail to hold for
 data taken from
  other functional spaces.
For instance, our proof of \rp{0gain1/2} does rely on the continuity of $\var$ and $\db\var$ up to $\ov D$;
in particular, we do not know whether $\dbar u=\var$ has an $L^\infty(D)$-solution when $\var$ is a $\db$-closed  form whose coefficients are merely in $L^\infty(D)$. The answer to such question would be positive if one knew that the closure of a strongly $\cc$-linearly convex domain $D$ of class $C^{1,1}$  can be exhausted by strongly pseudoconvex subdomains $\{D_j\}_j$ whose Levi forms are positive definite on the complex tangent spaces with bounds uniform in $j$, but
when $\pd D$ is merely $C^{1,1}$ it is not known whether $\ov D$ admits such an exhaustion; see
\rrem{lastrem}
in Section \ref{S:basics}.
It would be interesting to  understand the regularity of $\db$-solutions on a domain  whose  boundary is locally biholomorphic to  a strongly $\cc$-linearly convex $C^{1,1}$ real hypersurface
 (whose definition is embedded in the statement of ~\rp{hyper} below). However, it remains to be seen whether \rt{gain1/2}  extends to such domains.
\vskip0.08in

Finally, in the last section we derive ad-hoc estimates for
   the relevant counter-example in \cite{LSEx} indicating that some regularity of the $\dbar$-problem in the strong $\cc$-linearly convex category may persist below the class $C^{1,1}$;
 see Section \ref{S:example}
   for the precise statements.

 \vskip0.08in

The study of regularity of the solutions of the $\db$-problem  via integral representations has a long and rich history.
A detailed review of the existing literature may be found in ~\ci{MR847923} and, for the most recent results, ~\cite{MR3961327} and ~\ci{shi2019weighted}.
Here we briefly
recall that for smooth, strongly pseudoconvex domains the optimal
 $1/2$-estimate of
 \rp{0gain1/2} was achieved by Henkin-Romanov~\ci{MR0293121} for $\db$-closed
 forms
 after Grauert-Lieb~\ci{MR0273057}, Henkin~\ci{MR0249660}, Kerzman~\ci{MR0281944} proved that a $C^\beta$-estimate holds for any $\beta<1/2$. \rp{0gain1/2} for
 forms
  that are not
  necessarily
   $\db$-closed is due to Range-Siu~\ci{MR0338450}. The $C^{k+1/2}$ solutions for $\db u\in C^k$ were obtained by Siu~\ci{MR330515} for $q=1$ and by Lieb-Range~\ci{MR597825} for all degrees, and both require $\pd D\in C^{k+2}$ and $k\in\nn$.
The results in the aforementioned~\ci{MR3961327} were recently extended to weighted $L^p$ Sobolev spaces
by Shi~\ci{shi2019weighted}. A survey of the extensive literature on the solutions of the $\db$-problem with methods other than integral formulas may be found in e.g., Harrington~\cite{MR2491606}.
 \vskip0.08in

\noindent {\bf Acknowledgments.} Part of this work was carried out while the second-named author
was in residence at the Isaac Newton Institute for Mathematical Sciences during the program {\em Complex Analysis: techniques, applications and computations} (EPSRC grant no. EP/R04604/1).
We thank the Institute, and the program organizers, for the generous support and hospitality.
Finally, we are grateful to the referee for making several suggestions that have greatly improved the exposition.

\setcounter{thm}{0}\setcounter{equation}{0}
\section{More about $\cc$-linear convexity}\label{S:basics}

In this section we prove  Theorem \ref{equiv-cond}.  We will henceforth denote small positive constants by $c, c_1, c_*$, and large constants by $C,C_1,C_*$.    All these constants may depend on the choice of defining function
 $r$ for the domain $D$ when such an $r$ is involved. We will also deal with a neighborhood $U$ of $\ov D$ in which case the constants may also depend on $U$.

We recall the following stability property for the notion of strong (resp. strict) $\cc$-linear convexity, see also ~\cite{MR3250300}.
\begin{lemma}\label{r1r2} Let $r^{(j)}$, $j=1,2$,
be two defining functions for $D$ that are of class $C^1$ in neighborhoods $U_j$ of $\ov D$, $j=1, 2$. If $r^{(1)}$ satisfies condition
\eqref{cond-sC} $($resp., condition \eqref{c+0}$)$ with constant $c=c_1$, then
$r^{(2)}$ also satisfies condition \eqref{cond-sC} $($resp., condition \eqref{c+0}$)$ with constant $c=c_2$.
\end{lemma}
\begin{proof}
 As is well known, see Range~\cite{MR847923}*{Lemma II.2.5}, there is a positive and continuous function $h: U_1\cap U_2\to \mathbb R^+$ such that
$$
{r^{(2)}}
(z) = h(z)\, {r^{(1)}}
(z),\quad z\in U_1\cap U_2\, ,\quad \text{and}
$$
$$
d{r^{(2)}}(\zeta) = h(\zeta)\, d{r^{(1)}}(\zeta),\quad \zeta\in U_1\cap U_2\cap \pd D.
$$
It follows from the second condition above, that
$$
|\,
 r^{(2)}_\zeta\cdot(\zeta-z)| = h(\zeta)\,| r^{(1)}_\zeta\cdot(\zeta-z)|\quad \text{for all}\quad \zeta\in \pd D\ \ \text{and}\ \  \text{for all}\ z\in\cc^n,
$$
giving the desired conclusion.
\end{proof}

Thus {\em any} defining function of a  strongly (resp. strictly) $\cc$-linearly convex domain will satisfy condition \eqref{c+0}
 (resp. \eqref{cond-sC}).
 In particular, by Whitney extension, any strongly (resp. strictly) $\cc$-linearly convex domain
admits a defining function $r\in C^\infty(\cc^n\setminus \pd D)\cap C^{1}(\cc^n)$ that satisfies \eqref{c+0} (resp. \eqref{cond-sC}). Furthermore, if the regularity of $D$ is improved to class $C^{1,1}$, then $D$ will admit a defining function $r\in C^\infty(\cc^n\setminus \pd D)\cap C^{1, 1}(\cc^n)$) that satisfies condition \eqref{c+0} (resp. \eqref{cond-sC}).
 \vskip0.08in

 {\em Proof of Theorem \ref{equiv-cond}}.\quad We first verify  that \eqref{C+} is independent of the choice of the $C^{1,1}$-defining functions by proving
 an equivalent condition. We will use the notation
  $$d(z)=\dist^{E}(z,\pd D)$$
  where $\dist^{E}$ stands for Euclidean distance in $\cc^n$.
 \le{leray-bound}Let $D$ be a bounded domain with a $C^{1,1}$ defining function $r$ defined on a neighborhood $U$ of $\ov D$.
 Then
\eqref{C+} is equivalent to
 \eq{C++}
  | r_\zeta\cdot(\zeta-z)|\geq c(d(\zeta)+d(z)+|\IM (r_\zeta\cdot(\zeta-z))|+|\zeta-z|^2), \quad {z\in \ov{D}}, \quad \zeta\in U\setminus D,
 \eeq
 for some constant $c$ depending on the $C^{1,1}$-norm of $r$.
\ele
The proof below will in fact show that \eqref{C+} is equivalent to the (seemingly weaker) condition
\eq{C++short}
 | r_\zeta\cdot(\zeta-z)|\geq \tilde{c}\,(d(\zeta)+d(z)+ |\zeta-z|^2), \quad {z\in \ov{D}}, \quad \zeta\in U\setminus D,
\eeq
but here we choose to state the Lemma with condition \eqref{C++} as it is this particular formulation that is most frequently stated in the literature.
\vskip0.1in

\noindent {\em Proof of Lemma \ref{leray-bound}.}\ \  The implication \eqref{C++} $\Rightarrow$ \eqref{C+} is trivial. Suppose now that $r$ satisfies \eqref{C+},
and let
$\zeta\in U\setminus D$ and $z\in \ov D$.
 Since
\ga\label{C00}
2| r_\zeta\cdot(\zeta-z)|\geq |\RE (r_\zeta\cdot(\zeta-z))| + |\IM (r_\zeta\cdot(\zeta-z))|,\\
\label{zeta-z}
|\zeta-z|\geq d(z)+d(\zeta),\\
c\leq \f{-r(z)}{d(z)}\leq C, \quad \text{and} \quad c\leq \f{r(\zeta)}{d(\zeta)}\leq C,
\label{zeta-z+}
\end{gather}
it suffices to show that \eqref{C+} implies
\eq{C+++}
|r_\zeta\cdot(\zeta-z)|\geq c\big(r(\zeta)-r(z)+|\zeta-z|^2\big)
\eeq
for a (possibly different) constant $c>0$.

Under condition \eqref{C+} it is enough to verify \re{C+++} when   $|\zeta-z|$ is small.  By
Taylor formula, which is applicable because $r\in C^{1,1}$, we have
 \eq{2REr}
2\RE (r_\zeta\cdot(\zeta-z))\geq r(\zeta)-r(z)-C_0|\zeta-z|^2.
\eeq
Note that $r(\zeta)-r(z)>0$. If
\eq{key-case}
2C_0|\zeta-z|^2<r(\zeta)-r(z)
\eeq
 we obtain
$$
\RE r_\zeta\cdot(\zeta-z)\geq( r(\zeta)-r(z))/2\geq (r(\zeta)-r(z)+c'|\zeta-z|^2)/4.
$$
 Thus $|r_\zeta\cdot(\zeta-z)|\geq( r(\zeta)-r(z))/2\geq (r(\zeta)-r(z)+c'|\zeta-z|^2)/4$, which is equivalent to \re{C+++} for a possibly different constant $c$.
If
\eq{trivial}
2C_0|\zeta-z|^2\geq r(\zeta)-r(z),
\eeq
we
obtain
 by \eqref{C+}
$$
|
r_\zeta\cdot(\zeta-z)|\geq c|\zeta-z|^2\geq c(|\zeta-z|^2+\f{c_0}{4}(r(\zeta)-r(z)))/2,
$$
which also gives us \re{C+++} for a possibly different $c$.
\qed.

\le{leray-bound+}Let $D$ be a bounded domain of class $C^{1,1}$.
Then condition \eqref{C++} is independent of the choice of the $C^{1,1}$-defining function for $D$.
\ele
Here and in what follows
 we may shrink the neighborhoods $U$ of $D$ in \rea{C++} and \rea{C+}.
\begin{proof}  Let $ r\in C^{1,1}(U)$ and $\tilde r\in C^{1,1}(\tilde U)$ be two defining functions for $D$,
 and suppose that $r$ satisfies \eqref{C++}. We need to show that $\tilde r$ also satisfies \eqref{C++}, or equivalently that
\eq{C+1} | \tilde r_\zeta\cdot(\zeta-z)|\geq \tilde
c(d(\zeta)+d(z)+|\zeta-z|^2), \quad z\in D, \quad \zeta\in \tilde U\setminus D,
\eeq
where $\tilde U$ is some neighborhood of $\ov D$;  see \eqref{C++short}.
 But \rl{leray-bound} shows that  \eqref{C++} and \eqref{C+} are equivalent,   thus $r$ satisfies \eqref{C+} as well,  and we will use the latter to show that
$\tilde r$ satisfies
 \eqref{C+1} (and hence \eqref{C++}).

 We start with  $\tilde r=hr$ with $h\in Lip(U')$ and $h\in C^{1,1}_{loc}(U'\setminus\pd D)$, where $ U'$ is an open neighborhood of $\pd D$.
 Suppose that $\zeta\in U'\setminus D$.
 We have
\eq{trhr}
\tilde r_\zeta=hr_\zeta+rh_\zeta.
\eeq
We may assume that $c_0<h<C_0$ and (by Rademacher Theorem) that  $|h_\zeta|<
C_1$ for some $C_0$ and $C_1$. Then
$$
\tilde r_\zeta\cdot(\zeta-z)=h(\zeta)r_\zeta\cdot(\zeta-z)+r(\zeta)h_\zeta\cdot(\zeta-z).
$$
Combining this with condition \re{C+} applied to $r$, we have
\eq{temp}
|\tilde r_\zeta\cdot(\zeta-z)|\geq c_1 |\zeta-z|^2-C_1r(\zeta)|\zeta-z|.
\eeq
(To be precise, we obtain the above inequality first at those points $\zeta\in U'$ where the Lipschitz function $h$ is differentiable, and then we extend to any $\zeta\in U'$ by the continuity of $\tilde r_\zeta$.)
From this inequality and \re{zeta-z}-\re{zeta-z+}, we see that by
possibly
 shrinking $U$,
it suffices to verify \re{C+1} when $|\zeta-z|$ and hence $r(\zeta)$
are  sufficiently small. Invoking the elementary inequality:
$$2ab\leq \ \del^2 a^2+\del^{-2}b^2\quad  \text{for any}\ \  \del>0,$$
we see that
  $$
  r(\zeta)|\zeta-z|\leq \epsilon
  |\zeta-z|^2+\e^{-1}|r(\zeta)|^2.
$$ Without loss of generality, we may choose to make $U'$ so small  that $ r(\zeta)<
c_*\e$ for $\zeta\in\tilde U$ for a suitable
small
 $c_*>0$, which we reserve to choose later.

Plugging the above in \eqref{temp}  we obtain
\al\label{trz}
|\tilde r_\zeta\cdot(\zeta-z)|&\geq( c_1-
C_1\e) |\zeta-z|^2-
C_1c_*r(\zeta)
\geq c_2 |\zeta-z|^2-
C_1'
c_*\tilde r(\zeta)\\
&\geq
c_2 |\zeta-z|^2-
C_1'
c_*(\tilde r(\zeta)-\tilde r(z))\nonumber
\end{align}
where the last inequality is due to the fact that $z\in D$, that is $\tilde r(z)<0$.

We also have
\eq{trz+}
2|\tilde r_{\zeta}\cdot(\zeta-z)|\geq
2\RE(\tilde r_{\zeta}\cdot(\zeta-z))
\geq \tilde r(\zeta)-\tilde r(z)-C_2|\zeta-z|^2
\eeq
where the second inequality is obtained by applying Taylor's theorem to $\tilde r$ (see the proof of
 Lemma \ref{leray-bound}
 for a similar argument),
and we may assume without loss of generality that $C_2>1$.  We now choose $c_*$ (and thus $U'$) so that $c_2/{(c_*
C_1'
)}>
8
C_2$.
This gives us $$ |\tilde r_{\zeta}\cdot(\zeta-z)|\geq\f{c_2}{4}|\zeta-z|^2+\f{
C_1'
c_*}{2}(\tilde r(\zeta)-\tilde r(z)).$$
Here the inequality is obtained by
 invoking \re{trz} if $c_2|\zeta-z|^2>2
C_1'
c_*(\tilde r(\zeta)-\tilde r(z))$ or if not, by invoking \re{trz+}
to
 obtain
\aln
2|\tilde r_{\zeta}\cdot(\zeta-z)|&\geq \f{1}{2}(\tilde r(\zeta)-\tilde r(z)) +\f{1}{2}(\tilde r(\zeta)-\tilde r(z))-{C_2}|\zeta-z|^2\\
&\geq \f{1}{2}(\tilde r(\zeta)-\tilde r(z)) +\f{c_2}{4
C_1'
c_*}|\zeta-z|^2-{C_2}|\zeta-z|^2
\\
&\geq \f{1}{2}(\tilde r(\zeta)-\tilde r(z))  +\f{C_2}{2}|\zeta-z|^2,
\end{align*}
which proves \re{C+1}.
\end{proof}
\co{corollary} Condition \eqref{C+} is independent of the choice of $C^{1,1}$ defining function.
\eco
\begin{proof}
Let $r$ and $\tilde r$ be two defining functions for $D$, and suppose that $r$ satisfies condition \eqref{C+}. We proceed by contradiction and suppose that $\tilde r$ does not satisfy condition \eqref{C+}. Then $\tilde r$ does not satisfy condition \eqref{C++} (Lemma \ref{leray-bound}); but this in turn implies that $r$ does not satisfy condition \eqref{C++} (Lemma \ref{leray-bound+}). It follows by Lemma \ref{leray-bound} that $r$ does not satisfy condition
\eqref{C+}, giving a contradiction.
\end{proof}
The proof of Theorem \ref{equiv-cond} now continues with the following
\pr{C11case}If  $D$ is a bounded, strongly $\cc$-linearly convex domain of class $C^{1,1}$,
 then condition \eqref{c+0} is equivalent to \eqref{C+}
 for    $r\in C^{1,1}$.
 \epr
 \begin{proof} It's clear that \eqref{C+} $\Rightarrow$ \eqref{c+0}; we need to show the implication:  \eqref{c+0} $\Rightarrow$ \eqref{C+} i.e., that
  \eq{cond0+}
  |
  r_\zeta\cdot(\zeta-z)|\geq c_0|\zeta-z|^2
 \eeq
 holds for some positive $c_0$, $z\in  \ov D$ and all $\zeta\in   U\setminus D$ for some neighborhood $U$ of $\ov{D}$. By the assumed strong $\cc$-linear convexity of $D$,
  i.e.
  \eqref{c+0}, we have that the above holds for $z\in D$ and $\zeta\in\pd D$ for a possibly different constant $c_0$.

 By the continuity of $r_\zeta$, it is clear that  \eqref{cond0+} holds when
  $U\setminus D$ is sufficiently ``narrow'', which we assume, and when
   $\zeta\in U\setminus D,z\in D$ and $|\zeta-z|>c$, where $c$ is any positive constant and $U$ depends on $c$. Without loss of generality, we assume that $|\zeta-z|<c_*$ for a small $c_*$ to be determined.

Let $\zeta\in U\setminus\ov D$, $z\in D$, where $|\zeta-z|$ is sufficiently small, and let $\zeta_*\in\pd D$ be such that $$\dist^E(\zeta, \pd D)=|\zeta-\zeta_*|\, .$$
Since $D$ is, in particular, of class  $C^1$, we have that the line through $\zeta$ and $\zeta_*$ is perpendicular to $T_{\zeta_*}^{\mathbb R}(\pd D)$, the tangent space to $\pd D$ at $\zeta^*$.
By a translation and a unitary change of coordinates, we may assume that
 $$\zeta_*=0, \quad \text{and}\quad \zeta=(0',-i\la), \quad \la>0.
$$
Thus the real tangent space $T_0^{\mathbb R}(\pd D)$ is defined by $y_n=0$ and near the origin $D$ is defined by
\eq{rhat}
\hat r=-y_n+R(z',x_n)<0, \quad R(0', 0)=0,\ \nabla R(0', 0)=\vec{0},
\eeq
where $z'\in\cc^{n-1}$.
With the above choice of coordinates, we can easily relate $\hat r_\zeta\cdot(\zeta-z)$ to its value at $\zeta=\zeta_*$.
Set $z=(z',x_n+iy_n)$; we have
 \eq{split}
\hat r_{\zeta}\cdot (\zeta-z)=\f{1}{2i}(x_n+i(\la+y_n))=\hat r_{\eta}\cdot(\eta-z)|_{\eta=0}+\f{1}{2}\la.
 \eeq
Thus
\gan
|\hat r_{\zeta}\cdot (\zeta-z)|=\f{1}{2}|x_n+i(y_n+\la)|,\\
|\zeta-z|^2=|z'|^2+|x_n|^2+|y_n+\la|^2.
\end{gather*}
By the hypothesis and \rl{r1r2}, we have that $$\f{1}{2}|x_n+iy_n|=\big|\hat r_{\eta}\cdot(\eta-z)|_{\eta=0}\big|\geq c_1|z|^2.$$
This shows that when $x_n,y_n$ are sufficiently small, we have
\eq{xns}
|x_n|+|y_n|\geq c_1|z'|^2.
\eeq
 We claim that it suffices to show
 \eq{nocase-}
 |x_n|+|y_n+\la|\geq c_*|z'|^2.
 \eeq
 Indeed, assuming the truth of \re{nocase-}, we see that it and \re{split} give us
 \aln
 |\hat r_\zeta\cdot(\zeta-z)|&=\f{1}{2}|x_n+i(\la+y_n)|\\
 &
 \geq \f{1}{4}(|x_n+i(\la+y_n)|+c_*|z'|^2/2)\geq c|\zeta-z|^2.\end{align*}
 This gives us the required estimate in terms of $\hat r$.  By
  \rl{leray-bound+},
  we obtain \eqref{C+}
  for a possibly different $U$.
 \vskip0.08in

We are left to prove \re{nocase-}. We first give a simple argument in the case when $D$ is strongly convex.
Indeed,  for such $D$  we have
$
y_n>0.
$
We immediately obtain
$$
|x_n|+|y_n+\la|=|x_n|+|y_n|+\la\geq|x_n|+|y_n|\geq c_1|z'|^2,
$$
where the last inequality follows from \re{xns}.

We now consider the general case, and again
use
the original assumption \re{c+0} in the local coordinate system.
For any $z'$ as above, we momentarily consider an auxiliary point $\tilde z := (z', \tilde x_n +i \tilde y_n)$ by setting $\tilde x_n :=0$ and $\tilde y_n :=R(z',0)$. Then $\tilde z$
is in $\pd D$. Thus by \re{c+0}, we get
$$
\f{1}{2}|\tilde y_n|=|\hat r_\zeta\cdot(\zeta-\tilde z)|_{\zeta=0}|\geq c_1|\tilde z|^2\geq c_1|z'|^2.
$$
This shows that
\begin{equation}\label{Eqn:1}
|R(z',0)|\geq 2c_1|z'|^2, \quad\forall z'.
\end{equation}
We next make the stronger claim that
\begin{equation}\label{Eqn:2}
R(z',0)\geq 2c_1|z'|^2, \quad\forall z'.
\end{equation}
On account of \eqref{Eqn:1} just proved, it suffices to show that
$$ 
R(z',0)\geq 0,\quad  \forall z'.
$$ 
To this end, take any point $\zeta_*\in\pd D$. Let $\hat z=L_{\zeta_*}(z)=U(\zeta_*)(z-\zeta_*)$ be the composition of a translation and unitary transformation such that $
L_{\zeta_*}(\zeta_*)=0
$
and $L_{\zeta*}D$ is defined by
$$
-\hat y_n+R(\hat z',\hat x_n;\zeta_*)<0,
$$
where $R(0;\zeta_*)=0$ and $\nabla R(0;\zeta_*)=0$.  At the origin the complex tangent space of $L_{\zeta_*}D$ is defined by $\hat z_n=0$, which is uniquely determined by our construction. If $\tilde U(\zeta_*)$ is another choice that yields a different $\tilde R$, we   have
$$ 
-\hat y_n+\tilde R(\hat  z',\hat  x_n;\zeta_*)=(-\hat y_n+ R(V\hat z', \hat x_n;\zeta_*))h,
$$ 
where $V$ is a unitary matrix and $h$ is a positive function. Therefore,
the sign function
$$
S(\zeta_*)=\sign R(\hat z',0;\zeta_*)\in\{-1,1\}
$$
is independent of the choices of $L_{\zeta_*}$ and $ z'$ when $ z'\neq0$ is sufficiently small.

  Locally,    we can choose $L_{\zeta_*}$ depending on $\zeta_*\in\pd D$ continuously.  This shows that  $S$ is  well-defined and continuous  on $\pd D$. Next, we find a boundary point $\hat \zeta$ such that $D$ is located on one side of the real tangent space of $\pd D$  at $\hat \zeta$. For such $\hat \zeta$,  it is clear that $R(z',0;\hat\zeta)<0$ cannot occur for small $z'$. Thus $S$ is positive on the component of $\pd D$ containing $\hat\zeta$. But $\pd D$ (the boundary of the bounded domain of holomorphy $D\subset \cc^n$ with $n\geq 2$) must be connected by Hartogs's extension. It follows that $S>0$, thus proving \eqref{Eqn:2}.
\vskip0.1in
Now by the assumed $C^{1,1}$ regularity of  $R$ and the Taylor remainder theorem, we obtain
$$
|R(z',x_n)-R(z',0)-(\pd_{x^n}R(z',x_n))|_{x_n=0}x_n|\leq Cx_n^2, \quad \forall z',x_n.
$$
We have $\nabla R(0)=0$. When $(z',x_n)$ is sufficiently small (recall that $|\zeta-z|$ is small), we obtain
\eq{Rz'}
|R(z',x_n)-R(z',0)|\leq \e|x_n|+Cx_n^2.
\eeq
Here $\e>0$ can be made small by assuming that $\zeta, z',x_n$ are sufficiently  small.
 Recall that
$$
y_n>R(z',x_n),
$$
{by}
\eqref{rhat}. From the latter together with \eqref{Eqn:2} and \eqref{Rz'}, we see that
\al\label{ynR}
y_n&>R(z',x_n)=R(z',0)+(R(z',x_n)-R(z',0))\\
&\geq c_1|z'|^2-C|x_n|^2-\e|x_n|\geq c_1|z'|^2-\e'|x_n|,
\nonumber
\end{align}
where $\e'$ can be arbitrarily small by assuming $\zeta,z$ to be sufficiently close. Since $\la>0$, then
\eq{xnynL}
|x_n|+|y_n+\la|\geq |x_n|+y_n+\la\geq |x_n|+c_1|z'|^2-\e'|x_n|\geq c_1|z'|^2.
\eeq
Thus \re{nocase-} holds.

The above gives us the estimate for $\hat r_\zeta\cdot(\zeta-z)$;
  invoking
  \rl{leray-bound+}
   one more time, we conclude that $|r_\zeta\cdot(\zeta-z)|\geq c_2|\zeta-z|^2$ for a possibly  different $U$, thus concluding the proof of Proposition \ref{C11case}.
\end{proof}

 The following result allows us to introduce  a notion of
 strongly $\cc$-linearly convex real {\em hypersurface}.
\pr{hyper}Let $D$ be a bounded domain with a $C^{1,1}$ defining function $r$.
Then \eqref{c+0} and hence \eqref{C+}  are equivalent to \eqref{c}$:$
$$
|r_\zeta\cdot(\zeta-z)|\geq c|\zeta-z|^2, \quad \forall\zeta, z\in \pd D.
$$
\epr

Of course, condition \re{c} alone cannot tell which side of $\pd D$ is the domain $D$: it is the additional assumption of boundedness of $D$ that determines $D$.

\begin{proof}
Note that Proposition \ref{hyper} is meaningful also in the 1-dimensional setting, that is for $D\Subset\cc $, but in this case its conclusion is obvious and so in what follows we assume that
$D\Subset\cc^n$ with $n\geq 2$.
The direction \eqref{C+} $\Rightarrow$ \eqref{c} is trivial. We prove the opposite direction: suppose that \eqref{c} holds. We first show that $D$ is strictly $\cc$-linearly convex, i.e. that $r_\zeta\cdot(\zeta-z)\neq0$ for $\zeta\in\pd D$, $z\in\ov D\setminus\{\zeta\}$, by \eqref{cond-sC}. Indeed, suppose for the sake of contradiction that $r_\zeta\cdot(\zeta-z)=0$ for some  $\zeta\in \pd D$ and $z\in D$
(note that case $z\in\pd D$ is ruled out by \re{c}).

  Then $
H(\zeta)
:=\{z\in D\colon r_\zeta\cdot(\zeta-z)=0\}$ is non empty.  Since $H(\zeta)$ is a bounded domain in the complex hyperplane $\{z\in\cc^n\colon r_\zeta\cdot(\zeta-z)=0\}$
and   $n\geq 2$, the boundary of $H(\zeta)$
 contains more than one point and
 is a subset of $\pd D$.
But
 $z\in \pd H(\zeta)\setminus \{\zeta\}$ gives $r_\zeta\cdot(\zeta-
 z )=0$, which contradicts the assumption \eqref{c}. Thus $D$ is strictly $\cc$-linearly convex.

Since $r_\zeta\cdot(\zeta-z)\neq0$ for $\zeta\in\pd D$ and $z\in\ov D\setminus \{\zeta\}$, by \eqref{cond-sC}, then for any $\del>0$
$$
\inf\left\{|\zeta-z|^{-2}|r_\zeta\cdot(\zeta-z)|\colon \zeta\in\pd D,z\in \ov D,|\zeta-z|\geq\del\right\}>0.
$$
We now proceed to prove \re{C+}: we may assume without loss of generality that $z$ is in a small neighborhood of $\zeta\in\pd D$.
As in the proof of \rp{C11case}, by a unitary transformation and translation, we may assume that $\zeta=0$ and that near the origin $D$ is defined by
$$
y_n>R(z',x_n), \quad R(0', 0)=0,\ \ \nabla R(0', 0)=\vec{0}.
$$
Let $z = (z', x_n+iy_n)$ with $z', x_n$ and $y_n$ as above, and again consider the auxiliary point
$\tilde z:=(z',i\tilde y_n)$ where $\tilde y_n:=R(z',0)$.
Then $z$
satisfies \re{Rz'} and \re{ynR}.
Proceeding as in the proof of \re{xnynL} we  find
$$
|x_n|+|y_n|\geq
|x_n|+y_n\geq c|z'|^2.
$$
For $\hat r=-y_n+R(z',x_n)$, we obtain
$$
|\hat r_{\zeta}\cdot(\zeta-z)|=\f{1}{2}|x_n+iy_n| \geq\f{1}{4}(|x_n+iy_n|+\f{1}{2}(|x_n|+|y_n|))\geq\f{1}{4}(|x_n+iy_n|+
{\frac{c}{2}}
|z'|^2).
$$
We have verified \re{c+0} and hence \eqref{C+}.
\end{proof}

In the sequel we will need the following version of the Whitney extension theorem.
 \pr{C11-wh} Let $k\geq0$ be an integer.
Let $D$ be a bounded domain with Lipschitz boundary. If $f\in C^{k,\beta}$ with $0\leq\beta\leq1$, then   there exists an extension $E_kf\in C^{k,\beta}(\rr^N)$ such that $E_kf=f$ on $\ov D$ and
\eq{dxktr}
|\pd_x^\ell E_kf|\leq C_{k,\ell}(1+\dist(x,\pd D)^{k+\beta-\ell}), \quad x\in\cc^n\setminus \ov D,\quad \ell=0,1,\dots.
\eeq
\epr
For a proof, see~\cites{MR0101294, MR0290095, MR3961327}. The proof in ~\ci{MR3961327} assumes that $f\in C^{k+\all}(\ov D):=
C^{k,\all}(\ov D)$ with $0\leq\all<1$, but the same argument works for $f\in C^{k,1}(\ov D)$.
 \vskip0.08in

\begin{rem}\label{lastrem}
While the work carried out in this section here will be sufficient to prove the main result Theorem \ref{gain1/2}, one would like to know whether
 the closure of a strongly $\cc$-linearly convex domain $D$ whose boundary is merely of class $C^{1,1}$ has a basis of pseudo-convex neighborhoods; and whether such $D$ can be exhausted by strongly $\cc$-linearly convex $C^{1,1}$ domains that are relatively compact in $D$: at present we do not know the answers to such questions.
 \end{rem}

\section{Homotopy formulas}\label{S:homotopy}
 \setcounter{thm}{0}\setcounter{equation}{0}
\label{sect:hf}

In this section, we first derive a homotopy formulas for $C^{1,1}$ domains $D$ that are {\em weakly}
$\cc$-{\em linearly convex}, see \eqref{cond-wsC},
 in terms of the classical Leray-Koppelman operators $T_q$ which however need to be properly interpreted here, since the classical construction of $T_q$ requires two continuous derivatives of the (any) defining function of $D$ and these, in our context, are not available. These are the operators that occur in Proposition \ref{0gain1/2}. We mention in passing that the notion of $\cc$-linear convexity is also meaningful for domains below the $C^1$-category (and in this context such notion is often referred to simply as
 ``linear convexity'')
 but here it is  of no import.
 The interested reader is referred to~\ci{MR2060426} for a detailed discussion of this more general setting.
 \vskip0.05in
We next derive a homotopy formula for
bounded $C^{1,1}$ domains admitting a defining function $r$ that obeys the stronger condition
\eqref{strictconvU} in a neighborhood $U$ of $\ov D$ which, on account of Theorem \ref{equiv-cond}, is satisfied e.g, by strongly $\cc$-linearly convex $D$. Here the homotopy formulas are given in terms of operators $H_q$ that are constructed in Proposition~\ref{hf} below.

\subsection{The
Leray-Koppelman
homotopy operators
and homotopy formulas for
 weakly
  $\cc$-linearly convex
 domains of class $C^{1,1}$.
}
Let $D$ be a  bounded domain defined by a $C^{1,1}$ function $r$ defined on a bounded, open neighborhood $U$
of $\ov D$.
We first assume that
$r$ satisfies \eqref{cond-wsC}, which we recall here:
$$
|r_\zeta\cdot(\zeta-z)|>0, \quad \zeta\in\pd D, \quad z\in D\, .
$$
  Set
\eq{defg1}
g^0(z,\zeta) :=\ov\zeta-\ov z\,;\quad
g^1(z,\zeta)
:=r_\zeta\, \quad \text{and}\quad
w=\zeta-z.
\eeq
Note that while
$g^1(\zeta,z)$ does not depend on $z$,  we maintain this notation to conform with the literature for Cauchy-Fantappi\`e forms.  In particular $g^1$ is (trivially) holomorphic in $z$.

Also let
$$
V :=D\times (U\setminus D)\, ,\quad \text{and}\quad S:=\{(z,\zeta)\in V\colon r_\zeta\cdot(\zeta-z)=0\}.
$$
 \vskip0.08in
Note that $S$ is a  {(possibly empty)}
closed subset of $V$, and that

\eq{not0}
  g^0\cdot w\neq0,\quad \zeta\neq z; \qquad g^1\cdot w\neq0\quad (z,\zeta)\in V\setminus S.
\eeq

   We formally define
     \ga
\omega^i\ :=\ \f{1}{2\pi i}\f{g^i\cdot dw}{g^i\cdot w}\, ;
\quad
\Omega^i
(z, \zeta)
\ :=\ \omega^i\wedge(\ov\pd\omega^i)^{n-1},\quad i=0, 1,\quad
\text{and}\label{Om1}\\
\Omega^{01}
(z, \zeta)
\ :=\omega^0\wedge\omega^1\wedge\sum_{\alpha+\beta=n-2}
(\ov\pd\omega^0)^{\alpha}\wedge(\ov\pd\omega^1)^{\beta}, \quad (\zeta,z)\in V\setminus S.
\label{Om01}
\end{gather}
Here the implied variable $(z, \zeta)$ is taken in $V$, and  the differentials $d$ and $\db$ are taken with respect to both $z$ and $\zeta$.

Note that since the components of $\nabla r$, and hence of $g^1$, are only Lipschitz,
the formal definitions of $\Omega^1(z, \zeta)$ and of $\Omega^{01}(z, \zeta)$ given above
 must be suitably interpreted.
More precisely, let
 $g^1_k$
 be a
 smooth approximation
  of
$g^1$, see \eqref{defg1}, where $k\in\mathbb N$ is sufficiently large so that  \eqref{not0} is true with $g^1_k$ in place of $g^1$. Hence \eqref{Om1} and \eqref{Om01}
 with $g^1_k$ in place of $g^1$ give meaningful notions of $\Om^1_k(z, \zeta)$ and $\Om^{01}_k(z, \zeta)$. The Rademacher Theorem now ensures that the limits as $k\to\infty$ of $\Om^1_k(z, \zeta)$ and $\Om^{01}_k(z, \zeta)$ exist for every $z\in D$ and a.e. $\zeta\in U_{K}
 \setminus D$
 and are in $L^\infty(K\times (U_{K}
 \setminus D))$ for any compact subset $K\subset D$,
 and we take $\Om^1(z, \zeta)$ and $\Om^{01}(z, \zeta)$  to be such limits.
 Here $U_{K}$ is an open set containing $\ov D$.
 We have the following representations:

  \eq{Omi}
\Om^{i}(z, \zeta)\, =\,
\f{1}{(2\pi \sqrt{-1})^n} \frac{g^i\cdot d\zeta}{g^i\cdot(\zeta-z)}\wedge\left(\f{\db g^i\dot{\wedge}\, d\zeta}{{g^i\cdot(\zeta-z)}}\right)^{n-1},
\eeq
where we have adopted the shorthand: $\db g^i\dot{\wedge}\,d\zeta:=\db g^i_1\wedge d\zeta_1+\cdots+\db g^i_n\wedge d\zeta_n$,
and

\eq{Om11}
\Om^{01}
(z, \zeta)\, =\,
 \f{1}{(2\pi \sqrt{-1})^n}  \sum_{\all+\beta=n-2}   \frac{g^0\cdot d\zeta \wedge(\db g^0
 \dot{\wedge}\,d\zeta)^{\all}}{(g^0\cdot (\zeta-z))^{\all+1}}
  \wedge \frac{ g^1\cdot d\zeta\wedge(\db g^1\dot{\wedge}\,d\zeta)^\beta}{(g^1\cdot (\zeta-z))^{\beta+1}}.
\eeq
 Furthermore, we decompose
\eq{deco}
\Om^i(z, \zeta)=
\sum\limits_{q=0}^{n-1}
\Om_{0,q}^i(z, \zeta),\quad \text{and}\quad
\Om^{01}(z, \zeta)=
\sum\limits_{q=0}^{
n-2
}
\Om_{0,q}^{01}(z, \zeta).
\eeq
Here both $\Om_{0,q}^i
(z, \zeta)
$ and $\Om_{0,q}^{01}
(z, \zeta)
$
have
  type $(0,q)$  in the variable $z$; on the other hand, in the variable $\zeta$ the
   type of $\Om_{0,q}^i
  (z, \zeta)
  $ is $(n,n-1-q)$, while $\Om_{0,q}^{01}
  (z, \zeta)
  $ has type $(n,n-2-q)$. And we have set $\Om_{0,-1}^1
  (z, \zeta)
  :=0$ and
   $\Om^{01}_{0,-1}
   (z, \zeta)
   :=0$.
The previous argument gives that each  term
$\Om_{0,q}^i(z, \zeta)$ and
$\Om_{0,q}^{01}(z, \zeta)$
in the decompositions \eqref{deco} is in $L^\infty(K\times (U_K
\setminus D))$ for any compact set $K\subset D$.
  \vskip0.08in

   Next we formally define the {\em Leray-Koppelman operators}:
   \eq{Kop}
   T_q\var (z) :=-\!\!\!\!\int\limits_{\zeta \in \pd D}\!\!\!\Om_{0,q-1}^{01}(z,\zeta)\wedge \var (\zeta)+
  \!\! \int\limits_{\zeta \in D}\!\!\!\Om_{0,q-1}^0(z,\zeta)\wedge \var (\zeta), \  \ q=1,\ldots,n
   \eeq
   where $\var$ is a form of type $(0, q)$ whose coefficients are continuous on $\ov D$.
 In giving this definition we face a new conceptual difficulty again stemming from the hypothesis that
 $r$ is only of class $C^{1,1}(U)$: the Rademacher theorem grants that the second-order derivatives of $r$ are in $L^\infty(U)$, hence
 $\pd^2 r|_{\pd D}$ may be undefined on $\pd D$ as the latter has Lebesgue measure 0 in $\cc^n$. Thus the boundary integral in \eqref{Kop} is, in principle, problematic. However one can  show that
\eq{def-it}
\int_{\zeta\in \pd D}\Om^{01}_{0,q-1}(z, \zeta)\wedge\var(\zeta)
\eeq
is nonetheless meaningful. This can be verified e.g., by expressing $\pd D$ as the graph of a function $\psi\in C^{1,1}(\rr^{2n-1})$ and employing the argument in ~\cite{MR3250300} which we briefly recall here.
Using a partition of unity, we may assume that the coefficients of $\var$ have compact support in a small neighborhood $V$ of $\zeta_0\in\pd D$. On $V$, we assume that
$\pd D$ is given by
$$
y_n=\psi(z',x_n),
$$
where $\psi\in C^{1,1}(\mathbb R^{2n-1})$.
We take $r:=-y_n+\psi(z',x_n)$ and   invoke the $C^{1,1}$-regularity of $\psi$ to get a sequence of smooth functions
 $\{\psi^{(k)}\}_k$ that converges to $\psi$ in the $C^1$ norm, while $\pd^2\psi^{(k)}$ are uniformly bounded
  and, furthermore,
   $\pd^2\psi^{(k)}(\zeta)\to \pd^2\psi (\zeta)$ as $k\to\infty$, whenever $\zeta$ is
 a Lebesgue point  of
  $\pd^2\psi$.
  \vskip0.08in

 We now work with the following defining function for $D$:
\eq{goodr}
r :=-y_n+\psi(z',x_n)
\eeq
and with its smooth approximates
\eq{goodrk}
  r^{(k)}:=-y_n+\psi^{(k)}(z',x_n).
\eeq

We define the form $g^1$ and smooth approximations $\{g^1_k\}_k$ using this choice of $r$.
This gives an approximation of $\Om^{01}(z, \zeta)$ by smooth forms $\Om^{01}_k(z, \zeta)$
which admit a decomposition analogous to \eqref{deco} as a sum of $(0,q)$ forms
 $\Om^{01}_{(0,q)\, k}(z, \zeta)$ with smooth coefficients. Proceeding as in ~\cite{MR3250300} (Rademacher theorem in the variables $(z', x_n)\in\mathbb R^{2n-1}$) one can show that
   each coefficient of
 $\Om^{01}_{(0,q)\, k}(z, \zeta)$ converges a.e. $\zeta\in\pd D$ to a limit
 which is in
 $L^\infty(K\times \pd D)$ for any compact subset $K\subset D$,
 which must agree with (the corresponding coefficient of) the limit
$\Om_{0,q}^{01}(z, \zeta)$
that was previously determined. In short, we have that
\eq{Okq}
\Om_{0,q}^{01}(z, \zeta)\in
 L^\infty(K\times(U_K
 \setminus D))\cap
   L^\infty(K\times \pd D)
   \eeq
   for any compact subset $K\subset D.$
It follows that
$$
\int\limits_{\zeta\in \pd D}\Om^{01}_{(0,q-1)\, k}(z, \zeta)\wedge\var (\zeta)\  -
\int\limits_{\zeta\in \pd D}\Om_{0,q-1}^{01}(z,\zeta)\wedge\var (\zeta)\to 0
$$
uniformly on the compact subsets of $D$
as $k\to\infty$.
This shows that \eqref{def-it}, and hence $T_q$ is indeed well-defined for $r$ as in \re{goodr}.
   \vskip0.08in

   The above arguments also show that the conclusions of {\em Koppelman Lemma}:
  \ga
\label{kop1}\db_\zeta\Om_{0,q}^1
(z, \zeta)
+\db_z\Om_{0,q-1}^1
(z, \zeta)
=0,   \quad
q =0,\ldots,
n-1
\\
\label{kop2}\db_\zeta\Om^{01}_{0,q}
(z, \zeta)
+\db_z\Om_{0,q-1}^{01}
(z, \zeta)
=\Om_{0,q}^0
(z, \zeta)
-\Om_{0,q}^1
(z, \zeta),
\quad
q =0,\ldots,
 n-1
\end{gather}
are valid for our choice of $\Om_{0,q}^i,\  i=0, 1$ and $\Om^{01}_{0,q}$, $q=0,\ldots, n-1$, for
\eq{Sk-}
 (z,\zeta)\in V\setminus S= D\times (U\setminus D)\setminus\{(z,\zeta)\colon r_\zeta\cdot(\zeta-z)=0\}.
\eeq
Indeed,
by the classical Koppelman lemma~\cite{MR1800297}*{p.~263}, identities \eqref{kop1} and
\eqref{kop2} are  valid for $\Om^i_{(0, q)\, k}(z, \zeta), i=0, 1$ and $\Om^{01}_{(0, q)\, k}(z, \zeta)$ for any $k\in\mathbb N$ and
\eq{Sk}
(z,\zeta)\in V\setminus
S_k
:= D\times (U\setminus D)\setminus\{(z,\zeta)\colon r^{(k)}_\zeta\cdot(\zeta-z)=0\}.
\eeq
 Taking the limit as $k\to\infty$ we obtain \eqref{kop1} and \eqref{kop2}.
 \vskip0.08in

\le{Tq-stable}  Let $D\subset\cc^n$ be a bounded,
weakly $\cc$-linearly
 convex domain of class $C^{1,1}$. The definition of the Leray-Koppelman homotopy operators $T_q$,
i.e.
 \eqref{Kop},
 is independent of the choice of
 $C^{1,1}$
 defining function for $D$.
\ele
\begin{proof} Recall that $T_q\var$ was defined via a partition of unity for $\ov D$ and local graph defining function of $\pd D$. To verify that $T_q\var$ is independent of the choice of $r\in C^{1,1}$, we again use a partition of unity and the above local defining function. For any defining function $\tilde r\in C^{1,1}$ of $D$,
using \re{goodr}
we have
$$
\tilde r=hr, \quad \tilde r_\zeta=hr_\zeta, \quad \zeta\in\pd D.
$$
Here $h\in Lip$. We now compute $\Om^{01}$ defined by \re{Om11} with $g^1= r_\zeta$ being replaced by $\tilde g^1=\tilde r_\zeta$. We first note that since $\tilde r
\in C^{1,1}$  then in particular $\tilde r|_{\pd D}\in C^{1,1}$ and hence
$d(\tilde g^1\cdot d\zeta)\in L^\infty(\pd D)$;  see~\cite{MR3250300}*{Section~2}. Thus  $T_q\var$ is well-defined when $g^1$ is replaced by $\tilde g^1$.  With $h\in Lip$ additionally, we conclude that
on $\pd D$,
$$
\tilde g^1\cdot d\zeta\wedge (d_\zeta(\tilde g^1\cdot d\zeta))^\all=hg^1\cdot d\zeta\wedge (d_\zeta(hg^1\cdot d\zeta))^\all=h^{\all+1}g^1\cdot d\zeta\wedge (d_\zeta(g^1\cdot d\zeta))^\all.
$$
 We can also verify that the above expressions have $L^\infty$ coefficients,  and the identities hold
 in the sense of distributions. Approximating $\var\in C^0$ by $C^1$ functions, we conclude that $T_q\var$ is independent of the $\tilde r$. \end{proof}

\pr{hft}Let $D$ be a
bounded weakly
$\cc$-linearly
 convex domain with a defining function $r\in C^{1,1}$. Then
on $D$
\ga\label{chtf--}
\var(z)=\db_zT_q\var+T_{q+1}\db_z\var,  \quad 1\leq q\leq n,\\
\var(z)=\int_{\pd D}\Om_{0,0}^1\var +T_1\db \var,
\quad q=0,
\label{lerayf-}
\end{gather}
 hold for
for any $(0,q)$-form $\var\in C^0(\ov D)$ whose
distributional derivatives $\db \var$ on $D$ extend to a form whose coefficients are  in $C^0(\ov D)$.
\epr

\begin{proof} By the Whitney extension theorem, we may assume that $r\in  C^{\infty}(U\setminus\pd D)$.
By \rl{r1r2}, we have that such an $r$ satisfies condition \eqref{cond-sC}:
$$
 |r_\zeta\cdot(\zeta-z)|>0, \quad \zeta\in\pd D, \quad z\in
  D.
 $$
 Let $D'$ be any relatively compact subdomain of $D$.  We get   from the above that
 \eq{rz>0+}
 |r_\zeta\cdot(\zeta-z)|>c, \quad \zeta\in\pd D, \quad z\in D'.
 \eeq
We now define
$$D_j:= \{r<-{c'}
/j\}\quad \text{where}\ \ {c'}
 \ \ \text{is a small positive number}.
$$
Note that $D_j$ is relatively compact in $D$ and $D_j$ increase to $D$ as $j\to\infty$. Thus we may henceforth assume that $D'$ is relatively compact in each $D_j$ (for $j>j_0$),   and that \re{rz>0+} holds for $z\in D'$ and $\zeta\in\pd D_j$.

 As before, we take a sequence of smooth functions $\{r^{(k)}\}_k$ that tend to $r$ in the $C^1(\ov{D})$-norm, while $\nabla^2r^{(k)}$ are uniformly bounded, and $\partial^2 r^{(k)} \to  \partial^2r$ pointwise a.e. on a neighborhood of $\ov{D}$.
 Thus, replacing $r$ in \re{rz>0+} with $r^{(k)}$, we have
\eq{rkD'}
|r^{(k)}_\zeta\cdot(\zeta-z)|>0, \quad \zeta\in\pd D_j, \quad z\in D'.
\eeq

 We next let $\var_\e\in C^{\infty}(\ov D_j)$ be such that
 $\db\var_\e- \db\var$
tends to $0$ in {the}
sup norm on $\ov D_j$ as $\e\to0$, on account of the assumption that $\var\in C^0(\ov D)$
and $\db\var$ on $D$ extends to a continuous form on  $\ov D$.

 Let us consider the case of $q\geq1$. Assume that $\var$ and $\db\var$ are in $C^0(\ov D)$.  We first recall
 the Bochner-Martinelli-Koppelman formula for $C^1$ domains~\ci{MR1800297}*{Theorem 11.1.2}:
 \eq{BM-}
 \var(z)=\db_z\int_D\Om_{0,q-1}^0(z,\zeta)\wedge\var(\zeta)+\int_D\Om_{0,q}^0(z,\zeta)\wedge\db\var+\int_{\pd D}\Om_{0,q}^0(z,\zeta)\wedge\var(\zeta).
 \eeq
Fix $z\in D'$.
Applying \re{rkD'}
 to $D_j$ with $(z, \zeta)$ as in ~\re{Sk}, we obtain via
an (implicit) analog of \re{kop2}
\aln
\int_{\pd D_j}\Omega^1_k(z,\zeta)\wedge\var_\e(\zeta)&=\int_{\pd D_j}\left\{\db_\zeta\Omega_{(0,q),k}^{01}\right\}\wedge
\var_\e+\int_{\pd D_j}\left\{\db_z\Omega_{(0,q-1),k}^{01}\right\}\wedge
\var_\e\\
&=-\int_{\pd D_j}\Omega_{(0,q),k}^{01}\wedge\db_\zeta
\var_\e-\db_z\int_{\pd D_j}\Omega_{(0,q-1),k}^{01}\wedge
\var_\e.
\end{align*}

Note that $\db(\var_\e)=(\db\var)_\e$ on $D_j$ when $0<\e<\e_j$ for a sufficiently small positive $\e_j$.
Thus we have
$$
\var_\e=\db T^{j,k}_q\var_\e+T^{j,k}_q
(\db\var)
_\e,\quad \text{on $D'$ and for $\e<\e_j$}
$$
where $T_q^{j,k}$ is the Leray-Koppelman operator \eqref{Kop}
associated with $D_j$ and $r^{(k)}$, given by
 $$
T_q^{j,k}\var (z) =-\!\!\!\!\int\limits_{\zeta \in \pd D_j}\!\!\!\Om_{(0,q-1)\,, k}^{01}(z,\zeta)\wedge \var (\zeta)+
  \!\! \int\limits_{\zeta \in D_j}\!\!\!\Om_{(0,q-1), \, k}^0(z,\zeta)\wedge \var (\zeta), \  \ q=1,\ldots,n.
  $$
We first let $\e$ tend to $0$ and then we
let  $j$ tend to $\infty$. Then on $D'$ we have
$$
\var=\db T^{k}_q\var+T^{k}_q\db\var$$
 where
  $$
   T^k_q\var (z) :=-\!\!\!\!\int\limits_{\zeta \in \pd D}\!\!\!\Om_{(0,q-1),\,k}^{01}(z,\zeta)\wedge \var (\zeta)+
  \!\! \int\limits_{\zeta \in D}\!\!\!\Om_{(0,q-1)\,,k}^0(z,\zeta)\wedge \var (\zeta), \  \ q=1,\ldots,n.
  $$
Now the argument that
was used to define the quantity
\re{def-it} shows that $T_q\var$ is well-defined by taking $k\to\infty$
 in the expression above. We have obtained
\re{chtf--}.
The proof for \re{lerayf-}
 follows a similar strategy.
\end{proof}

\subsection{A new
 homotopy formula  for
 bounded strongly $\cc$-linearly convex $C^{1,1}$ domains}

We now prove the main result of this section by deriving a homotopy formula for a domain $D$ admitting a $C^{1,1}$ defining function $r$ satisfying
\re{strictconvU},
which we recall says
$$
r_\zeta\cdot(\zeta-z)\neq0, \quad\forall \zeta\in U\setminus D, z\in D.
$$

Note that by \rp{C11case},
this condition
is weaker than \re{C+}
and follows from the strong $\cc$-linear convexity of $D\in C^{1,1}$.
  \pr{hf}
Let $D\subset\cc^n$ be a  domain  with a defining function  $r$
which is of class  $C^{1,1}$ in a neighborhood $U$ of $\ov D$.  Let $g^0(z,\zeta)=\ov\zeta-\ov z$.  Let $g^1(z,\zeta)=r_\zeta$  satisfy
condition
\rea{strictconvU}.
 Let $\var$ be  a $(0,q)$-form in $\ov D$. Suppose that $ \var$ and $\db\var$ are in $C^1(\ov D)$. Then in $D$
\ga\label{tsqf}
\var=\db H_q\var+H_{q+1}\db\var, \quad 1\leq q\leq n,\\
\label{tsqf+}
\var=H_0\var+H_1\db \var, \quad q=0,
\intertext{with}
\label{Hqv}
H_q\var:=\int_{U}\Om_{0,q-1}^0\wedge E\var+\int_{U\setminus D}\Om_{0,q-1}^{01}\wedge[\db,E]\var,  \quad q>0,\\
H_0\var:=\int_{\pd D}\Om_{0,0}^1\var-\int_{U\setminus D}\Om_{0,0}^1\wedge E\db \var=\int_{U\setminus D}\Om_{0,0}^1\wedge [\db, E] \var.
 \label{H0f}
\end{gather}
Here, $E$ is an operator that extends $ \var$ and $\db\var$ to $C^1(U)$.
\epr
\begin{proof} The extension operator  $E$ was constructed in ~\ci{MR0290095}, where it was defined for functions. Here we define $E\var$ by applying the original operator $E$ component-wise to the coefficients of $\var$, which results in a form of the same type. We may assume without loss of generality that $E\var $ has compact support in  $U$, by using cut-off functions. Thus $E\var,E\db \var$ are in $C^1(U)$ and $[\db,E]\var, [\db,E]\db\var=\db E\db\var$ are in $C^0(U)$.

 Let us consider the case: $q\geq1$. We  recall again the Bochner-Martinelli-Koppelman formula \re{BM-}:
 \eq{BM}
 \var(z)=\db_z\int_D\Om_{0,q-1}^0(z,\zeta)\wedge\var(\zeta)+\int_D\Om_{0,q}^0(z,\zeta)\wedge\db\var+\int_{\pd D}\Om_{0,q}^0(z,\zeta)\wedge\var(\zeta).
 \eeq
Assume that $\var$ and $\db\var$ are in $C^1(\ov D)$.
 We want to rewrite the boundary integral in \re{BM}. Fix a relatively compact subdomain $D'$ of $D$ and let $z$ vary in $D'$. Then we have by \re{kop1}-\re{kop2} applied to $g^1(\zeta,z)=r^{(k)}_\zeta$,
\aln
\int_{\pd D}\Om_{0,q}^0(z,\zeta)\wedge\var(\zeta)&=\int_{\pd D} \db_z\Om_{(0,q-1),k}^{01}(z,\zeta)\wedge\var(\zeta)+\int_{\pd D}\db_\zeta\Om_{(0,q),k}^{01}(z,\zeta)\wedge\var(\zeta)\\
&={\ov\pd_z}
\int_{\pd D} \Om_{(0,q-1),k}^{01}(z,\zeta)\wedge\var(\zeta)+\int_{\pd D}\Om_{(0,q),k}^{01}(z,\zeta)\wedge\db_\zeta\var(\zeta).
\end{align*}

By
 \ci{MR3961327}*{(2.12)},  we have for $z\in D$ and $\var\in C^1(\ov  D)$,
\al\label{keyidsim}
&-\db\int_{\zeta\in\pd D}\Om_{(0,q-1),k}^{01}( z,\zeta )\wedge \var(\zeta)+\db \int_{\zeta\in D}\Om_{0,q-1}^{0}( z,\zeta )\wedge \var(\zeta)
 \\
 &
 \nonumber \qquad\qquad=\db\int_{U\setminus D}
\Om_{(0,q-1),k}^{01}(z,\zeta)\wedge\db E\var(\zeta) +
\db\int_{U}
\Om_{0,q-1}^{0}(z,\zeta)\wedge E\var(\zeta).
\end{align}
When $\db\var\in C^1(\ov D)$,
 proceeding as in the proof of  ~\ci{MR3961327}*{(2.13)} we employ  ~\ci{MR3961327}*{(2.11)} to obtain
\al\label{keyidsim+} -\int_{\pd D}\Om_{(0,q),k}^{01}\wedge\db \var &+\int_D\Om_{0,q}^{0}\wedge \db \var
=\int_{U\setminus D}
\Om_{(0,q),k}^{01}\wedge\db E\db \var
\\
&- \db
\int_{  {U}  \setminus D}
\Om_{(0,q-1),k}^{01}\wedge E\db \var-\int_{  { U}  \setminus D}
 \Om_{(0,q),k}^{1}\wedge E\db \var \nonumber\\
 & +\int_{  {U}  \setminus D}
\Om_{0,q}^{0}\wedge E\db\var +\int_{D }
\Om_{0,q}^{0}\wedge \db \var.
\nonumber  \end{align}
On the right-hand side, the first term can be written via the commutator as
 $\db E\db \var =(\db E-E\db)\db \var$.
 Since $q\geq1$,  the third term  is zero. The second term, when combined with the first  term  on the right-hand side of \re{keyidsim},
 gives us the desired commutator for $\var$.    Adding   \re{keyidsim}-\re{keyidsim+} yields \re{tsqf} in which $g^1$ is $r^{(k)}_\zeta$.
 Letting $k\to\infty$, we obtain \eqref{tsqf} on $D$.

This completes the proof of \eqref{tsqf}. The proof of \eqref{tsqf+} (that is the case: $q=0$) follows a similar strategy.
\end{proof}
We mention that the commutator $[\db, E]$  was first employed by Peters~\ci{MR1135535}, where $E$ is the Seeley extension when $\pd D$ is sufficiently smooth. The commutator has been useful in the construction of homotopy formulas in other settings; for instance, see Michel~\ci{MR1134587} and Michel-Shaw~\ci{MR1671846}.

When $\pd D\in C^2$ is strongly pseudoconvex, the statement analogous to \rp{hf} has been proved recently by Shi~\ci{shi2019weighted} when $\var,\db\var$ are in the Sobolev space $W^{1,1}(D)$.
 \vskip0.08in

\begin{proof}[Proofs of \rta{gain1/2} and \rpa{0gain1/2}]  The inequality \re{dxktr}, where $f$ is replaced by $r|_{\ov D}$,   ensures that $W(\zeta,z):= r_\zeta(\zeta)$ with $r:=Ef$ is a regularized Henkin-Ramirez function in the sense of
 ~\cite{MR3961327}*{Definition~4.1} where the requirement that $W\in C^1$ (see condition (i.) in ~\cite{MR3961327}*{Definition~4.1}) is relaxed to $W$ being Lipschitz.   Note that conditions \eqref{C++} and \eqref{dxktr} give that $\Phi(\zeta,z):=r_\zeta\cdot(\zeta-z)$ satisfies the hypotheses of ~\cite{MR3961327}*{Propositions 4.4 and 4.7}; see
 ~\cite{MR3961327}*{(4.4) and (4.5)}. Thus conclusions
  \re{theest} and  \re{thest++++} follow from ~\cite{MR3961327}*{Propositions 4.4 and 4.7}, respectively. With the same reasoning,  the proof of the well-known $1/2$-estimate of Henkin-Romanov~\cite{MR774049}*{Theorem 2.2.2} for strictly pseudoconvex $C^2$ domains is valid for \re{estTq} for our $C^{1,1}$ domain.
\end{proof}

\section{Below the $C^{1,1}$ category: one example}\label{S:example}
Let us  give an example of $C^{1,\all}$ domain which satisfies \eqref{C+}.
\begin{exmp}
\label{exA}
Consider  for $x\in\rr^n$
$$
f(x)=|x|^m, \quad 1<m<\infty.
$$
We claim
that for a given $C_0>0$
\eq{f(x)}
f(y)-f(x)-
\nabla f(x)\cdot
(y-x)\geq c|y-x|^{\max(m,2)} \ \  \text{whenever}\ \  |x|+|y|<C_0.
\eeq

Indeed, note that the estimate \eqref{f(x)} is trivial when $x$ or $y$ is the origin, so in the sequel we assume that  neither of $x,y$ is the origin.   We have
\aln
f(y)-f(x)-
\nabla f(x)\cdot
(y-x)&=|y|^m-|x|^m-m|x|^{m-2}x\cdot(y-x)\\
&=|y|^m+(m-1)|x|^m-m|x|^{m-2}x\cdot y.
\end{align*}
Thus the estimate with $f(y)-f(x)-
\nabla f(x)\cdot
(y-x)\geq c|y-x|^m$ holds trivially when $|x|<c_m|y|$, or $|y|<c_m|x|$, or $x\cdot y\leq0$.
Hence by dilation, we may assume that  $|x|=1$, $c<|y|<C$ and $x\cdot y\geq0$. By rotation, we know that the real Hessian of $|x|^m$ is positive definite at any point $x\neq0$. Consequently the restriction of $|x|^m$  to
 any real line segment $\theta\in[0,1]\mapsto (1-\theta)x+\theta y$
with
positive distance to the origin is a strictly convex smooth function of $\theta$.
\vskip0.08in

 Moreover, for $c_m<|y|<C_m$; $|x|=1$,  and $x\cdot y\geq0$, we have
$$
|(1-\theta)x+\theta y|^2=(1-\theta)^2|x|^2+\theta^2|y|^2+2\theta(1-\theta)x\cdot y\geq (1-\theta)^2|x|^2+\theta^2|y|^2\geq \f{|x|^2|y|^2}{|x|^2+|y|^2}.
$$
This shows that   $g(t):=|(1-t)x+ty|^m$ has $g''(t)\geq c_m,$ and thus $g(1)-g(0)-g'(0)\geq c_m/2$. This gives us that $$f(y)-f(x)-\nabla f(x)\cdot(y-x)\geq \tilde c_m|y-x|^2/2$$
and concludes the construction of Example \ref{exA}.
\end{exmp}
\noindent As a consequence of the above,
 given any collection $\{m_1,\ldots, m_k\}$ with $m_j>1, j=1,\ldots, k$, and setting
$$
f(x) :=\sum_{j=1}^k a_j|x^j|^{m_j}, \qquad x^j\in\rr^{n_j}, \quad a_j>0,\  |x^j|<C_0,
$$
where $x=(x^1,\dots, x^k)  \in\rr^{n_1+\cdots+n_k}$,
we have
$$
f(y)-f(x)-\nabla f(x)\cdot
(y-x)\geq c_m|y-x|^{\max(m_1,\dots, m_k,2)}.
$$
Here  $c_m$ depends only on $a_1,\dots, a_m$ and $C_0$.
In particular, if
\eq{convexr}
r(z) :=|x_1|^{m_1}+|y_1|^{m_2}+\cdots+|x_n|^{m_{2n-1}}+|y_n|^{m_{2n}},
\eeq
then
$$
r(z)-r(\zeta)-2\RE (r_\zeta\cdot(z-\zeta))\geq c|\zeta-z|^{\max(m_1,\dots, m_{2n},2)}, \quad |\zeta|+|z|<C,
$$
where $c$ depends on $m_1,\dots, m_{2n},C$.
Assume further that  $r(z)\leq r(\zeta)$. We obtain
$$
2\RE (r_\zeta\cdot(\zeta-z))\geq c'|\zeta-z|^{\max(m_1,\dots m_{2n},2)}.
$$
If $1<m_j\leq 2$ for all $j$, then the domains $\{r<C\}$ are strongly  $\cc$-linearly convex
 and of class $C^{1,\alpha}$ with $\alpha:= \min\{m_1, \ldots, m_{2n}\}-1$, when $C>0$. In fact condition \re{C+} holds  for such domains.

The above shows that the function $r$ in \re{convexr} is a convex function.  Thus a theorem of Dufresnoy~\ci{MR526786}  gives the following.
\pr{Dufresnoy} Let $q>0$ and let $r$ be given by \eqref{convexr}.  Suppose that $f\in C^\infty(\ov D)$ is
a $\dbar$-closed $(0,q)$ form on
$D:=\{r<c\}$. Then there is a solution $u\in C^{\infty}(\ov D)$ to $\dbar u=f$ on $D$.
\epr

On the other hand, for the above examples and, more generally for any $C^{1,\all}$ strongly $\cc$-linearly convex domain with $0<\all<1$, our integral formulas-based method for the solution of the $\ov\pd$-problem in the H\"older or Zygmund space category (in fact, {\em any} function space!) is conceptually problematic. This is because  the Leray-Koppelman homotopy operators $T_q$ cannot be made meaningful, as the components of $\nabla^2 r$ are not Lipschitz and there is neither a global, nor a tangential analog of the Rademacher Theorem that can be applied in this context.
 \vskip0.08in

\begin{rem}
One can check that condition \eqref{f(x)}
 is also valid for
$$
f(x) :=|x|^m g(x), \quad g\in C^{1,1}(\rr),\quad c<g (x)<C.
$$
 The  function above may be used to
 construct concrete examples of strongly
  $\cc$-linearly convex domains of class $C^{1,1}$ and no better.
\end{rem}

\newcommand{\doi}[1]{\href{http://dx.doi.org/#1}{#1}}
\newcommand{\arxiv}[1]{\href{https://arxiv.org/pdf/#1}{arXiv:#1}}

  \def\MR#1{\relax\ifhmode\unskip\spacefactor3000 \space\fi
  \href{http://www.ams.org/mathscinet-getitem?mr=#1}{MR#1}}

\bibliographystyle{abbrv}


\begin{bibdiv}
\begin{biblist}

\bib{MR2060426}{book}{
      author={Andersson, M.},
      author={Passare, M.},
      author={Sigurdsson, R.},
       title={Complex convexity and analytic functionals},
      series={Progress in Mathematics},
   publisher={Birkh\"{a}user Verlag, Basel},
        date={2004},
      volume={225},
        ISBN={3-7643-2420-1},
  url={https://doi-org.ezproxy.library.wisc.edu/10.1007/978-3-0348-7871-5},
      review={\MR{2060426}},
}

\bib{MR1512986}{article}{
      author={Behnke, H.},
      author={Peschl, E.},
       title={Zur {T}heorie der {F}unktionen mehrerer komplexer
  {V}er\"{a}nderlichen},
        date={1935},
        ISSN={0025-5831},
     journal={Math. Ann.},
      volume={111},
      number={1},
       pages={158\ndash 177},
         url={https://doi-org.ezproxy.library.wisc.edu/10.1007/BF01472211},
      review={\MR{1512986}},
}

\bib{MR1800297}{book}{
      author={Chen, S.-C.},
      author={Shaw, M.-C.},
       title={Partial differential equations in several complex variables},
      series={AMS/IP Studies in Advanced Mathematics},
   publisher={American Mathematical Society, Providence, RI; International
  Press, Boston, MA},
        date={2001},
      volume={19},
        ISBN={0-8218-1062-6},
      review={\MR{1800297}},
}

\bib{MR526786}{article}{
      author={Dufresnoy, A.},
       title={Sur l'op\'{e}rateur {$d^{\prime\prime}$} et les fonctions
  diff\'{e}rentiables au sens de {W}hitney},
        date={1979},
        ISSN={0373-0956},
     journal={Ann. Inst. Fourier (Grenoble)},
      volume={29},
      number={1},
       pages={xvi, 229\ndash 238},
         url={http://www.numdam.org/item?id=AIF_1979__29_1_229_0},
      review={\MR{526786}},
}

\bib{MR0101294}{article}{
      author={Glaeser, G.},
       title={\'{E}tude de quelques alg\`ebres tayloriennes},
        date={1958},
        ISSN={0021-7670},
     journal={J. Analyse Math.},
      volume={6},
       pages={1\ndash 124; erratum, insert to 6 (1958), no. 2},
         url={https://doi-org.ezproxy.library.wisc.edu/10.1007/BF02790231},
      review={\MR{0101294}},
}

\bib{MR3961327}{article}{
      author={Gong, X.},
       title={H\"{o}lder estimates for homotopy operators on strictly
  pseudoconvex domains with {$C^2$} boundary},
        date={2019},
        ISSN={0025-5831},
     journal={Math. Ann.},
      volume={374},
      number={1-2},
       pages={841\ndash 880},
         url={https://doi.org/10.1007/s00208-018-1693-9},
      review={\MR{3961327}},
}

\bib{MR0273057}{article}{
      author={Grauert, H.},
      author={Lieb, I.},
       title={Das {R}amirezsche {I}ntegral und die {L}\"{o}sung der {G}leichung
  {$\bar \partial f=\alpha $} im {B}ereich der beschr\"{a}nkten {F}ormen},
        date={1970},
        ISSN={0035-4996},
     journal={Rice Univ. Studies},
      volume={56},
      number={2},
       pages={29\ndash 50 (1971)},
      review={\MR{0273057}},
}

\bib{MR2491606}{article}{
      author={Harrington, P.S.},
       title={Sobolev estimates for the {C}auchy-{R}iemann complex on {$C^1$}
  pseudoconvex domains},
        date={2009},
        ISSN={0025-5874},
     journal={Math. Z.},
      volume={262},
      number={1},
       pages={199\ndash 217},
         url={https://doi.org/10.1007/s00209-008-0369-7},
      review={\MR{2491606}},
}

\bib{MR0249660}{article}{
      author={Henkin, G.M.},
       title={Integral representation of functions which are holomorphic in
  strictly pseudoconvex regions, and some applications},
        date={1969},
     journal={Mat. Sb. (N.S.)},
      volume={78 (120)},
       pages={611\ndash 632},
      review={\MR{0249660}},
}

\bib{MR774049}{book}{
      author={Henkin, G.M.},
      author={Leiterer, J.},
       title={Theory of functions on complex manifolds},
      series={Monographs in Mathematics},
   publisher={Birkh\"auser Verlag, Basel},
        date={1984},
      volume={79},
        ISBN={3-7643-1477-8},
      review={\MR{774049}},
}

\bib{MR0293121}{article}{
      author={Henkin, G.M.},
      author={Romanov, A.V.},
       title={Exact {H}\"{o}lder estimates of the solutions of the {$\bar
  \delta $}-equation},
        date={1971},
        ISSN={0373-2436},
     journal={Izv. Akad. Nauk SSSR Ser. Mat.},
      volume={35},
       pages={1171\ndash 1183},
      review={\MR{0293121}},
}

\bib{MR0281944}{article}{
      author={Kerzman, N.},
       title={H\"older and {$L^{p}$} estimates for solutions of {$\bar \partial
  u=f$} in strongly pseudoconvex domains},
        date={1971},
        ISSN={0010-3640},
     journal={Comm. Pure Appl. Math.},
      volume={24},
       pages={301\ndash 379},
         url={https://doi-org.ezproxy.library.wisc.edu/10.1002/cpa.3160240303},
      review={\MR{0281944}},
}

\bib{MR3084008}{article}{
      author={Lanzani, L.},
      author={Stein, E.M.},
       title={Cauchy-type integrals in several complex variables},
        date={2013},
        ISSN={1664-3607},
     journal={Bull. Math. Sci.},
      volume={3},
      number={2},
       pages={241\ndash 285},
  url={https://doi-org.ezproxy.library.wisc.edu/10.1007/s13373-013-0038-y},
      review={\MR{3084008}},
}

\bib{MR3250300}{article}{
      author={Lanzani, L.},
      author={Stein, E.M.},
       title={The {C}auchy integral in {$\Bbb{C}^n$} for domains with minimal
  smoothness},
        date={2014},
        ISSN={0001-8708},
     journal={Adv. Math.},
      volume={264},
       pages={776\ndash 830},
  url={https://doi-org.ezproxy.library.wisc.edu/10.1016/j.aim.2014.07.016},
      review={\MR{3250300}},
}

\bib{LSEx}{article}{
      author={Lanzani, L.},
      author={Stein, E.M.},
       title={The Cauchy-Leray integral: counter-examples to the $L^p$-theory},
        date={2019},
     journal={Indiana U. Math. J.},
      volume={68},
      number={5},
       pages={1609\ndash 1621},
}

\bib{MR597825}{article}{
      author={Lieb, I.},
      author={Range, R.M.},
       title={L\"osungsoperatoren f\"ur den {C}auchy-{R}iemann-{K}omplex mit
  {${\mathcal C}^{k}$}-{A}bsch\"atzungen},
        date={1980},
        ISSN={0025-5831},
     journal={Math. Ann.},
      volume={253},
      number={2},
       pages={145\ndash 164},
         url={https://doi.org/10.1007/BF01578911},
      review={\MR{597825}},
}

\bib{MR1134587}{article}{
      author={Michel, J.},
       title={Integral representations on weakly pseudoconvex domains},
        date={1991},
        ISSN={0025-5874},
     journal={Math. Z.},
      volume={208},
      number={3},
       pages={437\ndash 462},
         url={https://doi.org/10.1007/BF02571538},
      review={\MR{1134587}},
}

\bib{MR1671846}{article}{
      author={Michel, J.},
      author={Shaw, M.-C.},
       title={A decomposition problem on weakly pseudoconvex domains},
        date={1999},
        ISSN={0025-5874},
     journal={Math. Z.},
      volume={230},
      number={1},
       pages={1\ndash 19},
         url={https://doi.org/10.1007/PL00004685},
      review={\MR{1671846}},
}

\bib{MR1135535}{article}{
      author={Peters, K.},
       title={Solution operators for the {$\overline\partial$}-equation on
  nontransversal intersections of strictly pseudoconvex domains},
        date={1991},
        ISSN={0025-5831},
     journal={Math. Ann.},
      volume={291},
      number={4},
       pages={617\ndash 641},
         url={https://doi.org/10.1007/BF01445231},
      review={\MR{1135535}},
}

\bib{MR847923}{book}{
      author={Range, R.M.},
       title={Holomorphic functions and integral representations in several
  complex variables},
      series={Graduate Texts in Mathematics},
   publisher={Springer-Verlag, New York},
        date={1986},
      volume={108},
        ISBN={0-387-96259-X},
  url={https://doi-org.ezproxy.library.wisc.edu/10.1007/978-1-4757-1918-5},
      review={\MR{847923}},
}

\bib{MR0338450}{article}{
      author={Range, R.M.},
      author={Siu, Y.-T.},
       title={Uniform estimates for the {$\bar \partial $}-equation on domains
  with piecewise smooth strictly pseudoconvex boundaries},
        date={1973},
        ISSN={0025-5831},
     journal={Math. Ann.},
      volume={206},
       pages={325\ndash 354},
         url={https://doi.org/10.1007/BF01355986},
      review={\MR{0338450}},
}

\bib{shi2019weighted}{article}{
      author={Shi, Z.},
       title={Weighted sobolev ${L}^p$ estimates for homotopy operators on
  strictly pseudoconvex domains with ${C}^2$ boundary},
        date={2019},
     journal={ArXiv e-prints},
      eprint={1907.00264},
}

\bib{MR330515}{article}{
      author={Siu, Y.-T.},
       title={The {$\bar \partial $} problem with uniform bounds on
  derivatives},
        date={1974},
        ISSN={0025-5831},
     journal={Math. Ann.},
      volume={207},
       pages={163\ndash 176},
         url={https://doi-org.ezproxy.library.wisc.edu/10.1007/BF01362154},
      review={\MR{330515}},
}

\bib{MR0290095}{book}{
      author={Stein, E.M.},
       title={Singular integrals and differentiability properties of
  functions},
      series={Princeton Mathematical Series, No. 30},
   publisher={Princeton University Press, Princeton, N.J.},
        date={1970},
      review={\MR{0290095}},
}

\end{biblist}
\end{bibdiv}

\end{document}